\documentclass[12pt]{amsart}
\usepackage{graphicx} 
\usepackage{amssymb,latexsym,eucal}

\usepackage{hyperref}
\usepackage[letterpaper,centering,text={6.5in,9in}]{geometry}
\usepackage{latexsym}
\usepackage[utf8]{inputenc}
\usepackage{amsmath}
\usepackage{amsthm}
\usepackage{bigints}
\usepackage{amsfonts}
\usepackage{amssymb}
\usepackage{epsfig}
\usepackage{nicefrac}
\usepackage[all]{xy}
\usepackage{graphicx}
\usepackage{enumerate}
\usepackage{enumitem}
\usepackage{mathtools}    
\DeclarePairedDelimiterX\setc[2]{\{}{\}}{\,#1 \;\delimsize\vert\; #2\,}
\usepackage{float}
\usepackage{pgfplots}

\usepackage{caption,subcaption}
\usepackage{breqn}
\usepackage{amsaddr}
\usepackage{tikz}
\usetikzlibrary{
intersections, arrows.meta,
automata,er,calc,
backgrounds,
mindmap,folding,
patterns,
decorations.markings,
fit,decorations.pathmorphing,
shapes,matrix,
positioning,
shapes.geometric,
arrows,through, 
}
\usepackage{tikz-cd}

\def\bigmid{\ \rule[-3.5mm]{0.1mm}{9mm}\ }

\newtheorem{theorem}{Theorem}[section]
\newtheorem{corollary}[theorem]{Corollary}

\newtheorem{lemma}[theorem]{Lemma}

\theoremstyle{definition}

\newtheorem{remark}[theorem]{Remark}

\newcommand{\CC}{{\mathbb C}}

\newcommand{\RR}{{\mathbb R}}
\newcommand{\ZZ}{{\mathbb Z}}

\newcommand{\QQ}{{\mathbb Q}}

\newcommand{\ton}{{\otimes_{\Lambda_{\geq 0}}}}

\newcommand{\lr}{{\,\,\longrightarrow\,\,}}

\title{Comparison of symplectic capacities}
\author{Jonghyeon Ahn}
\newcommand{\Addresses}{{
\bigskip
\bigskip
\footnotesize
\textsc{Department of Mathematics,
University of Illinois Urbana-Champaign, Urbana, IL, 61801, USA.}\par\nopagebreak
\textit{E-mail address}: \texttt{ja34@illinois.edu}}}

\date{}

\begin{document}

\maketitle
\begin{abstract}
   In this paper, we compare the symplectic (co)homology capacity with the spectral capacity in the relative case. This result establishes a chain of inequalities of relative symplectic capacities, which is an analogue of the non-relative case. This comparison gives us a criterion for the relative almost existence theorem in terms of heaviness. Also, we investigate a sufficient condition under which the symplectic (co)homology capacity and the first Gutt-Hutchings capacity are equal in both non-relative and relative cases. This condition is less restrictive than the dynamical convexity. \\
\end{abstract}
\bigskip
\section{Motivation and results}
A key challenge in symplectic geometry is to determine when one symplecitic manifold can be symplectically embedded into another. Gromov asserts in his seminal paper \cite{gro} the ``non-squeezing theorem" saying that the ball
\begin{align*}
    B^{2n}(r) = \left\{ z \in \CC^n \bigmid \pi|z|^2 \leq r \right\}
\end{align*}
can be symplectically embedded into the cylinder
\begin{align*}
    Z^{2n}(R) = \left\{ z = (z_1,\cdots,z_n) \in \CC^n \bigmid \pi|z_1|^2 \leq R \right\}
\end{align*}
if and only if $r \leq R$. Inspired by this marvelous statement, we can obtain a nontrivial obstruction to the existence of symplectic embeddings which is more subtle than the volume restriction: A \textbf{symplectic capacity} $c$ assigns to symplectic manifold $(M, \omega)$ a number $c(M, \omega) \in [0, \infty]$ satisfying
\begin{itemize}
    \item (Monotonicity) if there exists a symplectic embedding $\phi : (M, \omega) \hookrightarrow (M', \omega')$, then $c(M,\omega) \leq c(M', \omega')$, and\\
    \item (Conformality) if $\lambda > 0$, then $c(M, \lambda \omega) = \lambda c(M, \omega)$.\\
    
\end{itemize}
 We say that a symplectic capacity $c$ is \textbf{normalized} if it satisfies
\begin{itemize}
    \item $c\left(B^{2n}(1), \omega_{2n}\right) = c(Z^{2n}(1), \omega_{2n}) = 1$\\
\end{itemize}
where $\omega_{2n}$ is the standard symplectic form in $\CC^n = \RR^{2n}$. We usually drop the symplectic form $\omega$ from the notation of the symplectic capacity $c(M,\omega)$ whenever it is clear from the context. One immediate example of normalized symplectic capacity from the non-squeezing theorem is the \textbf{Gromov width} $c^{Gr}(M)$ given by
\begin{align*}
    c^{Gr}(M) = \sup\left\{ r \bigmid B^{2n}(r)\,\,\text{can be symplectically embedded into}\,\,M\right\}.
\end{align*}
There are many examples of normalized symplectic capacities such as the Hofer-Zehnder capacity $c^{HZ}$ defined in \cite{hz}, symplectic (co)homology capacity $c^{SH}$ introduced in \cite{fhw, vit}, the first Gutt-Hutchings capacity $c^{GH}_1$ established in \cite{gh}. 
\\

One major question that encompasses the normalized symplectic capacities, which was recently refuted by Ostrover and Haim-Kislev \cite{ohk}, is (strong) \textit{Viterbo's conjecture}: If $K$ is a convex domain in $\RR^{2n}$, then all normalized symplectic capacities of $K$ are equal. The paper \cite{ghr} provides a comprehensive exposition of relevant results pertaining to this question. Although disproved in general, it still leaves many interesting questions. In this paper, we investigate specific symplectic capacities for which the equality asserted by Viterbo's conjecture could potentially hold.\\

As a fundamental first step in establishing the equality, we require inequalities. Significant progress has been made in deriving inequalities that relate the \textit{$\pi_1$-sensitive Hofer-Zehnder capacity} $\tilde{c}^{HZ}(K)$, the \textit{spectral capacity} $c^S(K)$ and the \textit{symplectic (co)homology capacity} $c^{SH}(K)$ for a Liouville domain $K$. We summarize the results below.
\begin{theorem}[\cite{bk, fgs,i}] Let $K$ be a Liouville domain. Then
    \begin{align}\label{nonrelch}
        \widetilde{c}^{HZ}(K) \leq c^S(K) \leq c^{SH}(K)
    \end{align}
\end{theorem}

 A \textbf{relative symplectic capacity} $c$ assigns to each triple $(M, K, \omega)$ for a  symplectic manifold $(M, \omega)$ and a subset $K \subset M$ a number $c(M, K, \omega) \in [0, \infty]$ that satisfies
\begin{itemize}
    \item (Monotonicity) if there exists a symplectic embedding $\phi : (M, \omega) \hookrightarrow    (M', \omega')$ such that $\text{int}(\phi(K)) \subset K'$, then $c(M, K, \omega) \leq c(M', K', \omega')$, and\\
    \item (Conformality) if $\lambda > 0$, then $c(M, K, \lambda \omega) = \lambda c(M, K, \omega)$.\\

    \end{itemize}
The symplectic capacities appearing in \eqref{nonrelch} each possess a relative counterpart, which we denote by $\widetilde{c}^{HZ}(M,K)$, $ c^S(M,K)$ and $c^{SH}(M,K)$, respectively. The actual statement in \cite{fgs} concerns the comparison $\widetilde{c}^{HZ}(M,K)$ and $ c^S(M,K)$ demonstrating the inequality:
\begin{align}\label{partineq}
    \widetilde{c}^{HZ}(M,K) \leq c^S(M,K).
\end{align}
To derive a corresponding chain of inequalities in the relative setting, we require an additional inequality connecting $c^S(M,K)$ and $c^{SH}(M,K)$.
\begin{theorem}[Theorem \ref{mcomparison}]\label{mcompint}
     Let $(M,\omega)$ be a closed symplectically aspherical symplectic manifold and $K\subset M$ be a Liouville domain with index-bounded boundary. Then
    \begin{align*}
        c^S(M,K) \leq c^{SH}(M,K)
    \end{align*}
    and therfore, together with \eqref{partineq}, 
    \begin{align*}
        \widetilde{c}^{HZ}(M,K) \leq c^S(M,K) \leq c^{SH}(M,K).
    \end{align*}
\end{theorem}
A consequence of Theorem \ref{mcompint} concerns the ``almost existence theorem". In \cite{gg}, Ginzburg and G{\"u}rel proved the relative almost existence theorem.

\begin{theorem}[\cite{gg}]
    Let $(M,\omega)$ be a symplectically aspherical closed symplectic manifold and $K \subset M$ be a compact nonempty subset. Let $H : M \to \RR$ be a proper smooth function on $M$ such that $H|_K = \min H$. If $c^{HZ}(M,K)$ is finite, then for almost all (in the sense of measure theory) regular values $c$ in the range of $H$, the level set $H^{-1}(c)$ carries a closed characteristic. 
\end{theorem}
In view of relative almost existence theorem and Theorem \ref{mcompint}, the finiteness of $c^{SH}(M,K)$ implies the same conclusion. In fact, stronger conclusions can be attainable with the aid of following statements.

\begin{theorem}[\cite{msv}]\label{heavy} Let $(M, \omega)$ be a closed symplectic manifold and $K \subset M$ be a compact subset. Then the subset $K$ is heavy if and only if $SH_M(K;\Lambda) \neq 0$. 
   
\end{theorem}
In Theorem \ref{heavy}, the notion of \textit{heaviness} is a type of symplectic rigidity introduced by Entov and Polterovich in \cite{ep}. The precise definition will be provided in \S 2.
\begin{theorem}[\cite{a}]
    Let $(M, \omega)$ be a closed symplectically aspherical symplectic manifold and $K \subset M$ be a Liouville domain with index-bounded boundary. If $SH_M(K;\Lambda) = 0$, then the relative symplectic (co)homology capacity $c^{SH}(M,K)$ is finite.
\end{theorem}
A criterion for the relative almost existence theorem is given below, based on the notion of heaviness.
\begin{corollary}
    Let $(M,\omega)$ be a closed symplectically aspherical symplectic manifold and $K\subset M$ be a Liouville domain with index-bounded boundary. Let $H : M \to \RR$ be a proper smooth function on $M$ such that $H|_K = \min H$. If $K$ is not heavy, then for almost all regular values $c$ in the range of $H$, the level set $H^{-1}(c)$ carries a closed characteristic.  \\ \qed
\end{corollary}

Our attention now shifts to an equality, in particular, between the first Gutt-Hutchings capacity $c^{GH}_1(M,K)$ and the symplectic (co)homology capacity $c^{SH}(M,K)$. It is straightforward from the definition that $$c_1^{GH}(M,K) \leq c^{SH}(M,K).$$ Recall that a $(2n-1)$-dimensional contact manifold $(C, \xi, \alpha)$ with $c_1(\xi) = 0$ is called \textbf{dynamically convex} if every Reeb orbit $\gamma$ on $(C, \alpha)$ satisfies $\mu_{\text{CZ}}(\gamma) \geq n+1$, where  $\mu_{\text{CZ}}(\gamma)$ denotes the Conley-Zehnder index of $\gamma$. Under this convexity assumption, the reverse inequality also holds, leading to the following conclusion.

\begin{theorem}[\cite{a}]\label{dycon}
    Let $(M,\omega)$ be a symplectically aspherical closed symplectic manifold and $K \subset M$ be a Liouville domain with index-bounded boundary. Let $\alpha$ be the canonical contact form on $\partial K$. If $(\partial K, \alpha)$ is dynamically convex, then
    \begin{align*}
        c_1^{GH}(M,K) = c^{SH}(M,K).
    \end{align*}
\end{theorem}
In this paper, we prove the same result given in Theorem \ref{dycon} under a less restrictive assumption on the Conley-Zehnder indices.
\begin{theorem}[Theorem \ref{slight}]\label{czint}
     Let $(M,\omega)$ be a $2n$-dimensional symplectically aspherical closed symplectic manifold and $K\subset M$ be a Liouville domain with index-bounded boundary. If every contractible Reeb orbit $\gamma$ of $(\partial K, \alpha)$ satisfies $\mu_{\text{CZ}} (\gamma) \geq n$, then
     \begin{align*}
          c_1^{GH}(M,K) = c^{SH}(M,K).
     \end{align*}
\end{theorem}
    Indeed, in Theorem \ref{mcompint}, Theorem \ref{dycon} and Theorem \ref{czint}, the constraint that $(M, \omega)$ is a closed symplectic manifold can be relaxed. For a Liouville domain $K$, let $\widehat{K}$ be the symplectic completion of $K$. Then $c^S(\widehat{K},K)$, $c^{GH}_1(\widehat{K},K)$ and $c^{SH}(\widehat{K},K)$ still make sense and they recover their non-relative counterparts, that is, 
    \begin{align*}
        c^S(\widehat{K},K) = c^S(K), \,\,c^{GH}_1(\widehat{K},K) =  c^{GH}_1(K)\,\,\text{and}\,\,c^{SH}(\widehat{K},K) = c^{SH}(K).
    \end{align*}
    Background of this extension can be found in \cite{sun}. Moreover, we can remove the index-boundedness assumption on $\partial K$ in the non-relative case essentially because Hamiltonian functions in the non-relative case are chosen to be linear at infinity. 

\begin{corollary}[Corollary \ref{nonrela}]\label{star}
   Let $K$ be a $2n$-dimensional Liouville domain with $c_1(TK)=0$. If every contractible Reeb orbit $\gamma$ of $(\partial K, \alpha)$ satisfies $\mu_{\text{CZ}} (\gamma) \geq n$, then
    \begin{align}\label{new}
        c_1^{GH}(K) = c^{SH}(K). 
    \end{align}
\end{corollary}
\begin{remark}
\begin{enumerate}[label=(\alph*)]
    \item A \textbf{star-shaped domain} $K$ in  $\RR^{2n}$ is a compact $2n$-dimensional submanifold such that the boundary $\partial K$ is transverse to the radial vector field
\begin{align*}
    X = \frac{1}{2}\sum_{i=1}^n \left( x_i \frac{\partial}{\partial x_i} + y_i \frac{\partial }{\partial y_i} \right)
\end{align*}
where $(x_1,\cdots,x_n,y_1,\cdots,y_n)$ are the coordinates on $\RR^{2n}$. Note that every star-shaped domain is a Liouville domain. In this case, the 1-form
\begin{align*}
    \frac{1}{2}\sum_{i=1}^n (x_i dy_i - y_i dx_i)
\end{align*}
restricted to $\partial K$ is a contact form on $\partial K$. So, Corollary \ref{star} can be restated as: For a star-shaped domain $K \subset \RR^{2n}$ with smooth boundary, if every contractible Reeb orbit on $\partial K$ has Conley-Zehnder index greater than or equal to $n$, then $ c_1^{GH}(K) = c^{SH}(K)$.
\item For a convex star-shaped domain $K \subset \RR^{2n}$ with smooth boundary, the agreement \eqref{new} of two symplectic capacities is known and the proof is a combination of many papers: Let $\mathcal{A}_{\min}(\partial K)$ be the minimal period of Reeb orbits of $\partial K$. It is proved in \cite{gh} and \cite{gs} that $$ c_1^{GH}(K) = \mathcal{A}_{\min}(\partial K).$$ Also, it is proved in \cite{ak} and \cite{i22} that $$ c^{SH}(K) = \mathcal{A}_{\min}(\partial K).$$

\item Hofer, Zehnder and Wysocki proved in their landmark paper \cite{hzw} that if $K \subset \RR^{4}$ is a strictly convex star-shaped domain with smooth boundary, then its boundary $\partial K$ is dynamically convex. It is explained in \cite{dgz, dgrz} that every convex domain in $\RR^4$ with smooth boundary has dynamically convex boundary. Therefore, Corollary \ref{star} is a slight generalization of the agreement $c_1^{GH}(K)= c^{SH}(K)$ to a wider class of star-shaped domains in $\RR^4$.  
\end{enumerate}
     
\end{remark}
\bigskip
\noindent
\textbf{Outline of the paper.} In \S 2, we provide a brief background on several versions of Floer cohomology. Also, the definitions of relative symplectic capacities $c^{HZ}(M,K)$, $c^{S}(M,K)$, $ c_1^{GH}(M,K)$ and $c^{SH}(M,K)$ can be found there. In \S 3, we prove our main results. The proofs of Theorem \ref{mcompint} and Theorem \ref{czint} can be found in \S 3.1 and \S3.2, respectively.\\

\medskip
\noindent
\textbf{Acknowledgement.} The author wishes to express sincere gratitude to Ely Kerman for his insightful discussions and constructive feedback.
\bigskip
\section{Preliminaries}
\subsection{Hamiltonian Floer theory}
In this subsection, we briefly review Hamiltonian Floer theory and related notions. For further details, see \cite{ep, g, o, sala, s}.
\subsubsection{Floer cohomology}
 The \textbf{Novikov field} $\Lambda$ is defined by
\begin{align*}
    \Lambda = \left\{ \sum_{i=1}^\infty c_i T^{\lambda_i} \bigmid c_i \in \QQ, \lambda_i \in \RR\,\,\text{and}\,\, \lim_{i \to \infty} \lambda_i = \infty \right\} 
\end{align*}
where $T$ is a formal variable. There is a valuation map $val : \Lambda \to \RR \cup \{\infty\}$ given by
\begin{align*}
    val (x) = 
    \begin{cases}
      \displaystyle\min_{i} \{\lambda_i \mid c_i \neq 0 \} \,&\text{if}\,\, x = \displaystyle\sum_{i=1}^\infty c_i T^{\lambda_i} \neq 0\\
      \infty \,&\text{if}\,\, x = 0.
    \end{cases}   
\end{align*}
For any $r \in \RR$, define $\Lambda_{\geq r} = val^{-1}([r,\infty])$. In particular, we call
\begin{align*}
    \Lambda_{\geq 0} = \left\{ \sum_{i=1}^\infty c_i T^{\lambda_i} \in \Lambda \bigmid  \lambda_i \geq 0 \right\}
\end{align*}
the \textbf{Novikov ring}. 
\par Let $(M, \omega)$ be a closed symplectic manifold and let $H : S^1 \times M \to \RR$ be a Hamiltonian function on $M$. The \textbf{Hamiltonian vector field} $X_H$ of H is defined by 
\begin{align*}
    \iota_{X_H} \omega = dH.
\end{align*}
We say a Hamiltonian function $H :S^1 \times M \to \RR$ is \textbf{nondegenerate} if every 1-periodic orbit of $X_H$ is nondegenerate, that is, the Poincar\'e return map has no eigenvalue equal to 1. We denote the set of all nondegenerate contractible 1-periodic orbits of $X_H$ by $\mathcal{P}(H)$. Let $x \in \mathcal{P}(H)$ and $\widetilde{x}$ be a disk capping. The \textbf{action} of $(x, \widetilde{x})$ is defined by
\begin{align*}
    \mathcal{A}_H(x, \widetilde{x}) = \int_{D^2} \widetilde{x}^* \omega + \int_{S^1} H(t, x(t))dt.
\end{align*}
We can associate to a pair $(x, \widetilde{x})$ an integer using the \textit{Conley-Zehnder index}, which we denote by $\mu_{\text{CZ}}(x,\widetilde{x})$. Define an equivalence relation on the set of pairs of orbits and cappings by
\begin{align*}
    (x, \widetilde{x}) \sim (y, \widetilde{y}) \,\,\text{if and only if}\,\, x = y, \,\mathcal{A}_H(x, \widetilde{x}) = \mathcal{A}_H(y, \widetilde{y}) \,\,\text{and}\,\, \mu_{\text{CZ}}(x,\widetilde{x}) = \mu_{\text{CZ}}(y,\widetilde{y}).
\end{align*}
We denote the equivalence class of $(x, \widetilde{x})$ by $[x,\widetilde{x}]$ and the set of all equivalence classes by $\widetilde{\mathcal{P}}(H)$. The \textbf{Floer complex of $H$} is defined by
\begin{align*}
    CF(H) = \left\{\sum_{i=1}^{\infty}c_i [x_i, \widetilde{x_i}] \bigmid c_i \in \QQ,  [x_i, \widetilde{x_i}] \in  \widetilde{\mathcal{P}}(H)\,\,\text{and}\,\, \lim_{i \to \infty} \mathcal{A}_H([x_i, \widetilde{x_i}]) = \infty \right\}.
  \end{align*}
The grading of $CF(H)$ is given by
\begin{align*}
    | [x, \widetilde{x}]| = \mu_{\text{CZ}}([x, \widetilde{x}]).
\end{align*}
  Fix a generic almost complex structure $J$ on the tangent bundle $TM$ of $M$. For $x, y \in \mathcal{P}(H)$, let $\pi_2(M, x, y)$ be the set of homotopy classes of smooth maps from $\RR \times S^1$ to $M$ which are asymptotic to $x$ and $y$ at $-\infty$ and $\infty$, respectively. Consider, for $[x, \widetilde{x}], [y, \widetilde{y}] \in \widetilde{\mathcal{P}}(H)$ and $A \in \pi_2(M,x,y)$, the moduli space $\mathcal{M}(H,J; [x, \widetilde{x}], [y, \widetilde{y}]; A)$ of Floer  trajectories of $H$ connecting $x$ and $y$. More precisely, $\mathcal{M}(H,J; [x, \widetilde{x}], [y, \widetilde{y}]; A)$ is the set of smooth maps $u : \RR \times S^1 \to M$ satisfying the following.
\begin{itemize}
    \item The homotopy class of $u$ represents $A \in \pi_2(M,x, y)$.\\
    \item $\displaystyle\frac{\partial u}{\partial s} + J(u) \left( \frac{\partial u}{\partial t} - X_{H_t}(u) \right) = 0$.\\
    \item $\displaystyle\lim_{s \to -\infty} u(s,t) = x(t)$ and $\displaystyle\lim_{s \to \infty} u(s,t) = y(t)$.\\
    \item $[y, \widetilde{y}] = [y, \widetilde{x}\,\#\, (-u)] = [y, \widetilde{x}\,\#\, (-A)] $ where $\#$ denotes the connected sum.\\
\end{itemize}
Note that if $u \in \mathcal{M}(H,J; [x, \widetilde{x}], [y, \widetilde{y}]; A)$, then $ \mathcal{A}_H([x, \widetilde{x}]) \leq \mathcal{A}_H([y,\widetilde{y}])$ and $\mu_{\text{CZ}}([x, \widetilde{x}]) \leq \mu_{\text{CZ}}([y, \widetilde{y}])$. The \textbf{Floer differential} $d : CF^*(H) \lr CF^{*+1}(H)$ of the Floer complex $CF(H)$ is given by
\begin{align*}
    d [x, \widetilde{x}] = \sum_{\substack{[y, \widetilde{y}] \in \widetilde{\mathcal{P}}(H) \\ A \in \pi_2(M, x, y)}} \#\mathcal{M}(H ; [x, \widetilde{x}], [y, \widetilde{y}]; A)  [y, \widetilde{y}]
\end{align*}
where $\#\mathcal{M}(H ; [x, \widetilde{x}], [y, \widetilde{y}]; A)$ is the virtual count defined in \cite{p} by Pardon. It is well-known that $d^2 = 0$ and the \textbf{Floer cohomology} $HF(H)$ of $H$ is defined to be 
\begin{align*}
    HF^*(H) = H^* \left( CF(H), d \right).
\end{align*}
Let $CF^{> L}(H)$ be the subset of $CF(H)$ generated by $[x,\widetilde{x}]$ with $\mathcal{A}_H([x, \widetilde{x}]) > L$. Since Floer differential increases the action, $CF^{> L}(H)$ is a subcomplex of $CF(H)$. Hence, we can define
\begin{align*}
    HF^{> L, *} (H) = H^* (CF^{> L}(H), d).
\end{align*}
For two Hamiltonian functions $H_1$ and $H_2$ with $H_1 \leq H_2$, choose a monotone homotopy $H_s$ of Hamiltonian functions connecting them. Then there is a \textbf{continuation map} 
\begin{align*}
   c^{H_1, H_2} : CF(H_1) \lr CF(H_2)
\end{align*}
defined by the suitable count of Floer trajectories of $H_s$ connecting the orbits of $H_1$ and $H_2$. This map induces a map $HF(H_1) \lr HF(H_2)$. Note that continuation maps increase the actions and respect the gradings. The Floer cohomology $HF(H;\Lambda)$ of $H$ with coefficient in $\Lambda$ is defined by
\begin{align*}
    HF^*(H;\Lambda) = H^*\left( CF(H) \ton \Lambda \right) = HF^*(H)\ton \Lambda.
\end{align*}

\subsubsection{Spectral invariant} Let $(M, \omega)$ be a closed symplectic manifold and let $H : S^1 \times M \to \RR$ be a nondegenerate Hamiltonian function.  In \cite{pss}, Piunikhin, Salamon and Schwarz introduced a PSS isomorphism of $H$
\begin{align*}
    \text{PSS}^H : QH(M ; \Lambda) \lr HF(H ; \Lambda)
\end{align*}
where $QH(M; \Lambda) = H^*(M ; \ZZ) \otimes_{\ZZ} \Lambda$ is the quantum cohomology of $M$ equipped with quantum product. For each $a \in QH(M ; \Lambda)$, the \textbf{spectral invariant} of $a$ is defined by
\begin{align*}
    \rho(a ; H) = \sup \left\{ L \in \RR \bigmid \pi^{H,\leq L} \left(PSS^H(a) \right) = 0 \right\}
\end{align*}
where the map $\pi^{H, \leq L} : HF(H; \Lambda) \lr HF^{\leq L}(H; \Lambda)$ is induced by the projection map
\begin{align*}
    CF(H) \lr CF(H) / CF^{>L}(H).
\end{align*}
Among many properties of spectral invariants, we point out the \textit{Lipschitz property}:
\begin{align}\label{lp}
    \int_{S^1} \min_{M} (H_1 - H_2) dt \leq \rho(a; H_1) - \rho(a; H_2) \leq \int_{S^1} \max_{M} (H_1 - H_2) dt.
\end{align}
By the property \eqref{lp}, we can extend the definition of spectral invariants to all continuous functions $H : S^1 \times M \to \RR$ by $C^0$-approximation. 

Using this spectral invariant, we can define many useful notions. A subset $K \subset M$ is called $a$-\textbf{heavy} if $\rho(a ; H) \leq \displaystyle\max_{K} H$ for all $H \in C^{\infty}(M)$. If $a=1_M$ is the unit element of $QH(M;\Lambda)$, then we say that $K$ is a heavy set. Also, the spectral invariant can be used to define a symplectic capacity. For a compact subset $K \subset M$, we can define the \textbf{spectral capacity} $c^S(M,K)$ of $K$ inside $M$ by
\begin{align*}
    c^S(M, K) = \sup \left\{ - \rho(1_M; H) \bigmid H \leq 0 \,\, \text{and compactly supported in}\,\,S^1 \times (K -\partial K) \right\}.
\end{align*}

\medskip

\subsection{Relative symplectic cohomology}
In this subsection, we give a brief background about relative symplectic cohomology introduced by Varolgunes in \cite{v, vt}. More exposition and application can be found in \cite{a, dgpz, msv}.

 Let $(M, \omega)$ be a closed symplectic manifold and let $H : S^1 \times M \to \RR$ be a nondegenerate Hamiltonian function. Define the \textbf{weighted Floer complex} $CF_w(H)$ of $H$ by $CF_w(H) = CF(H)$. The reason why it is called weighted is that its differential is weighted by the \textit{topological energy} of a Floer trajectory. The \textbf{weighted Floer differential} $d_w : CF^*_w(H) \lr CF^{*+1}_w(H)$ of $CF_w(H)$ is defined by
\begin{align*}
    d_w [x, \widetilde{x}] = \sum_{\substack{[y, \widetilde{y}] \in \widetilde{\mathcal{P}}(H) \\ A \in \pi_2(M, x, y)}} \#\mathcal{M}(H,J ; [x, \widetilde{x}], [y, \widetilde{y}]; A) T^{E_{\text{top}}(u)} [y, \widetilde{y}].
\end{align*}
where the \textbf{topological energy} $E_{\text{top}}(u)$ of a Floer trajectory $u$ is given by
\begin{align*}
    E_{\text{top}}(u) = \int_{S^1} H_t(y(t)) dt - \int_{S^1} H_t(x(t)) dt + \omega(A).
\end{align*}
It can be easily shown that $$E_{\text{top}}(u) = \mathcal{A}_H([y,\widetilde{y}]) - \mathcal{A}_H([x, \widetilde{x}]).$$ For two Hamiltonian functions $H_1$ and $H_2$ with $H_1 \leq H_2$, we can construct a \textbf{weighted continuation map} 
\begin{align*}
    c_w^{H_1, H_2} : CF_w(H_1) \lr CF_w(H_2)
\end{align*}
using monotone homotopy $H_s$ connecting $H_1$ and $H_2$. This  map is also weighted by $T^{E_{\text{top}}(u)}$ in a similar way that we define the weighted differential above.
\begin{remark}\label{morse}
\begin{enumerate}[label=(\alph*)]
    \item We can define the weighted Floer cohomology $HF_w(H)$ and the weighted Floer cohomology $HF_w(H;\Lambda)$ with coefficient in $\Lambda$ of $H$ by
    \begin{align*}
        &HF^*_w(H) = H^*\left(CF_w(H)\right)\,\,\text{and}\\
        &HF^*_w(H;\Lambda) = H^*\left(CF_w(H) \ton \Lambda\right) = HF^*_w(H) \ton \Lambda.        
    \end{align*}
    But we have $$HF_w(H;\Lambda) \cong HF(H ; \Lambda)$$ because the map
    \begin{align*}
        CF(H) \ton \Lambda \lr CF_w(H) \ton \Lambda, \,\,[x,\widetilde{x}] \mapsto T^{\mathcal{A}_H([x, \widetilde{x}])}[x, \widetilde{x}]
    \end{align*}
    induces an isomorphism between $HF(H;\Lambda)$ and $HF_w(H;\Lambda)$.
    \item Let $f : M \to \RR$ be a Morse function. The Morse complex $CM(f)$ of $f$ in degree $k$ is defined by
    \begin{align*}
        CM^k(f) = \bigoplus_{\substack{ x \in \text{Crit}(f)\\ \text{ind}(x ;f) =k}} \Lambda_{\geq 0} \langle x \rangle
    \end{align*}
    where Crit($f$) is the set of all critical points of $f$ and ind($x ;f$) is the Morse index of $x \in \text{Crit}(f)$ with respect to $f$. And the differential of $CM(f)$ is defined by counting Morse trajectory $\gamma : \RR \to M$ satisfying 
    \begin{align*}
        \frac{d\gamma}{ds} =  \nabla_g f(\gamma(s))
    \end{align*}
    where $\nabla_g f$ is the gradient vector field of $f$ with respect to the chosen Riemannian metric $g$ on $M$. Note that a Morse trajectory increases the Morse index. We denote the Morse cohomology $HM^*(f)$ of $f$ by 
    \begin{align*}
        HM^*(f) = H^* \left( CM(f) \right).
    \end{align*}
     We can similarly define the weighted Morse complex. For a Morse function $f$, the weighted Morse differential is weighted by $T^{f(y) - f(x)}$ for Morse trajectory connecting from one critical point $x$ to another critical point $y$. But the map 
    \begin{align*}
        CM(H; \Lambda) \lr CM_w(H) \ton \Lambda, \,\,x \mapsto T^{f(x)} x.
    \end{align*}
    induces an isomorphism between the weighted Morse homology with coefficient in $\Lambda$ and the usual (unweighted) Morse homology with coefficient in $\Lambda$.
    \item If $H : S^1 \times M \to \RR$ is a $C^2$-small nondegenerate Hamiltonian function on $M$, then every Floer trajectory becomes a Morse trajectory of some Morse function $H'$, where the metric $g$ is defined by $$g(v,w) = \omega(v,Jw).$$ For $x \in \text{Crit}(f)$, let $\widetilde{x}$ be the constant disk capping of $x$. Since the relation between the Conley-Zehnder index and the Morse index is given by $$\mu_{\text{CZ}}([x, \widetilde{x}]) = \text{ind}(x, H) - n= n - \text{ind}(x, -H),$$ we have an isomorphism
    \begin{align*}
        HF^k(H) \cong HM_{n-k}(-H).
    \end{align*}
    \end{enumerate}

\end{remark}
To define the relative symplectic cohomology, we need one more input. For any module $A$ over $\Lambda_{\geq 0}$, we can define its completion as follows: For $r' >r$, there exists a map 
\begin{align*}
    \Lambda_{\geq 0} / \Lambda_{\geq r'} \lr \Lambda_{\geq 0} / \Lambda_{\geq r} 
\end{align*}
and this map induces a map
\begin{align}\label{inverse}
    A \ton \Lambda_{\geq 0} / \Lambda_{\geq r'} \lr A \ton \Lambda_{\geq 0} / \Lambda_{\geq r}.
\end{align}
Then $\left\{ A \ton \Lambda_{\geq 0} / \Lambda_{\geq r} \right\}_{r>0}$ is an inverse system with the maps given by \eqref{inverse}. The \textbf{completion} $\widehat{A}$ of $A$ is defined by
\begin{align*}
    \widehat{A} = \varprojlim_{r \to 0} A \otimes \Lambda_{\geq 0} / \Lambda_{\geq r}.
\end{align*}
There exists a natural map over $\Lambda_{\geq 0}$
\begin{align}\label{completion}
    A \lr \widehat{A}.
\end{align}
We say that $A$ is \textbf{complete} if the map \eqref{completion} is an isomorphism. 
Let $K \subset M$ be a compact subset. We say that a nondegenerate Hamiltonian function $H$ is \textbf{K-admissible} if $H$ is negative on $S^1 \times K$. We denote the set of all $K$-admissible Hamiltonian functions by $\mathcal{H}_K$. We can define the \textbf{relative symplectic cohomology} $SH_M(K)$ of $K$ in $M$ by
\begin{align*}
    SH^*_M(K) = H^* \left( \widehat{\varinjlim_{H \in \mathcal{H}_K}} CF_w(H) \right)
\end{align*}
where $\displaystyle\widehat{\varinjlim_{H \in \mathcal{H}_K}} CF_w(H)$ denotes the completion of $\displaystyle\varinjlim_{H \in \mathcal{H}_K} CF_w(H)$. Also, define 
\begin{align*}
    SH^*_M(K; \Lambda) = SH^*_M(K)\ton \Lambda.
\end{align*}

\medskip

\subsection{$S^1$-equivariant relative symplectic cohomology}
Here, we give a concise review of $S^1$-equivariant relative symplectic cohomology introduced in \cite{a}, which basically combines the definition of $S^1$-equivariant symplectic cohomology given by Bourgeois and Oancea in \cite{bo} and the definition of relative symplectic cohomology given by Varolugunes in \cite{v,vt}.

Let $(M, \omega)$ be a closed symplectic manifold and let $K \subset M$ be a compact domain with \textbf{contact type boundary}, that is, there exists an outward pointing vector field $X$ on a neighborhood of $\partial K$ such that $\mathcal{L}_X \omega = \omega$. We call the vector field $X$ a \textbf{Liouville vector field}. If the vector field $X$ is defined on $K$, then we call $K$ a \textbf{Liouville domain}. Note that $\alpha := \iota_X \omega|_{\partial K}$ is a contact form on $\partial K$. A small tubular neighborhood $\partial K \times (-\delta,\delta)$ endowed with a symplectic form $d\left( e^\rho \alpha \right)$ can be symplectically embedded in $M$ via the flow of the Liouville vector field. Here, $\rho$ is a coordinate on $(-\delta, \delta)$.

We say that a nondegenerate Hamiltonian function $H : S^1 \times K \to \RR$ is \textbf{contact type $K$-admissible} if it satisfies the following.
\begin{itemize}
    \item $H$ is negative on $S^1 \times K$.\\
    \item There exists $\eta \geq 0$ such that $H(t, p, \rho) $ is $C^2$-close to $ h_1(e^{\rho})$ on $S^1 \times \left(\partial K \times [0, \frac{1}{3}\eta]\right)$ for some strictly convex and increasing function $h_1$. \\
        \item $H(t, p, \rho) = \beta e^{\rho} +\beta'$ on $S^1 \times (\partial K \times [\frac{1}{3}\eta, \frac{2 }{3}\eta])$ where $\beta \notin \text{Spec}(\partial K, \alpha)$ and $\beta' \in \RR$. Here, $\text{Spec}(\partial K, \alpha)$ denotes the set of all periods of contractible Reeb orbits of $(\partial K, \alpha)$.\\
        \item $H(t, p, \rho) $ is $C^2$-close to $ h_2(e^{\rho})$ on $S^1 \times \left(\partial K \times [\frac{2 }{3}\eta, \eta]\right)$ for some strictly concave and increasing function $h_2$.\\
        \item $H$ is $C^2$-close to a constant function on $S^1 \times \left(M - (K \cup \left(\partial K \times [0, \eta]\right)\right))$.\\
\end{itemize}
 We denote the set of all contact type $K$-admissible Hamiltonian functions by $\mathcal{H}^{\text{Cont}}_K$. If a Hamiltonian function $H$ is $h(e^{\rho})$ on $S^1 \times( {\partial K} \times[0, \eta])$, then any nonconstant Hamiltonian orbit of $H$ is given by 
\begin{align}\label{hamorbit}
    X_H(p, \rho) = -h'(e^{\rho}) R(p)
\end{align}
where $R$ is the Reeb vector field of $(\partial K, \alpha)$. Hence, it corresponds to a Reeb orbit of $(\partial K, \alpha)$ of period $h'(e^{\rho})$ and traversed in the opposite direction of the Reeb orbit. If $K$ is a Liouville domain, then the action of a nonconstant Hamiltonian orbit $x$ of $H$ can be computed as
\begin{align}\label{for}
    \mathcal{A}_H(x) = - e^{\rho} h'(e^{\rho} ) + h (e^{\rho}). 
\end{align}

Let $H \in \mathcal{H}^{\text{Cont}}_K$ and define the complex $CF_w^{S^1}(H)$ by
\begin{align*}
    CF_w^{S^1}(H) = \Lambda_{\geq 0}[u] \otimes_{\Lambda_{\geq 0}} CF_w(H)
\end{align*}
where $u$ is a formal variable of degree 2. The grading of $CF_w^{S^1}(H)$ is given by
\begin{align}\label{grcon}
    | u^k \otimes [x,\widetilde{x}] | = -2k +\mu_{\text{CZ}}([x,\widetilde{x}]).
\end{align}Its differential is of the form 
\begin{align*}
    d_w^{S^1} (u^k \otimes [x,\widetilde{x}]) = \sum_{i=0}^k u^{k-i} \otimes \psi_i([x,\widetilde{x}]) 
\end{align*}
where $\psi_i$ is a count of Floer trajectories of some perturbation of $H$ weighted by the topological energy of Floer trajectories. For the precise definition of $\psi_i$, see \cite{bo}. The \textbf{$S^1$-equivariant relative symplectic cohomology of $K$ in $M$} is defined by
    \begin{align*}
        SH^{S^1,*}_M(K) &= H^* \left(\widehat{\varinjlim_{H \in \mathcal{H}^{\text{Cont}}_K}} CF_w^{S^1}(H)\right)\,\,\text{and}\,\,SH^{S^1,*}_M(K;\Lambda) =  SH^{S^1,*}_M(K) \ton \Lambda.
\end{align*}
\medskip
\subsection{Examples of relative symplectic capacities}
One example of relative symplectic capacity which we already saw is the spectral capacity explained in \S  2.1.2. In this subsection, we will introduce a couple of more relative symplectic capacities.

In \cite{gg}, Ginzburg and G{\"u}rel introduced a relative version of Hofer-Zehnder capacity which we briefly recall here. Let $(M, \omega)$ be a compact symplectic manifold possibly with boundary and $K \subset M$ be a compact subset. A smooth function $H : M \to \RR$ is called \textbf{Hofer-Zehnder admissible relative to $K$} if it satisfies the following.
\begin{enumerate}[label=(B\arabic*)]
    \item $H \leq 0$. 
    \item There exists an open subset $U \subset M- \partial M$ containing $K$ such that $H |_U = \min H$.
    \item There exists a compact subset $C \subset M - \partial M$ such that $H|_{M-C} = 0$.
    \item Every nonconstant periodic orbit of the Hamiltonian vector field $X_H$ has period $T > 1$.
\end{enumerate}
We denote the set of all Hofer-Zehnder admissible functions relative to $K$ on $M$ by $\mathcal{HZ}(M, K)$. Then the \textbf{relative Hofer-Zehnder capacity} $c_{HZ}(M,K)$ of $(M,K)$ is defined to be
\begin{align*}
    c^{HZ}(M, K) = \sup \left\{ - \min H \bigmid H \in \mathcal{HZ}(M,K) \right\}.
\end{align*}
In a similar way, we can define the $\pi_1$-sensitive relative Hofer-Zehnder capacity by replacing the axiom (B4). 
\begin{enumerate}[label=(B4$^{\circ}$)]
    \item Every nonconstant \textit{contractible} periodic orbit of the Hamiltonian vector field $X_H$ has period $T > 1$.
\end{enumerate}
Let $\mathcal{HZ}^{\circ}(M,K)$ be the set of all smooth functions on $M$ satisfying (B1) -- (B3) and (B4$^{\circ}$). The \textbf{$\pi_1$-sensitive relative Hofer-Zehnder capacity} $\Tilde{c}^{HZ}(M,K)$ of $(M,K)$ is defined by
\begin{align*}
    \Tilde{c}^{HZ}(M,K) = \sup \left\{ - \min H \bigmid H \in \mathcal{HZ}^{\circ}(M,K) \right\}.
\end{align*}
Note that $c^{HZ}(M, K) \leq \Tilde{c}^{HZ}(M,K)$. Furthermore, $c^{HZ}(M,K) \leq c^{HZ}(K)$ and $\Tilde{c}^{HZ}(M,K) \leq \Tilde{c}^{HZ}(K)$.

For the rest of this subsection, we assume that $(M, \omega)$ is a closed symplectic manifold which is \textit{symplectically aspherical}, that is,
\begin{align*}
    \omega|_{\pi_2(M)} = 0\,\,\text{and}\,\,c_1(TM)|_{\pi_2(M)} = 0
\end{align*}
where $c_1(TM)$ is the first Chern class of the tangent bundle $TM$ of $M$. Let $K \subset M$ be a Liouville domain. Additionally, we need the following assumption on $\partial K$: We say that $\partial K$ is \textbf{index-bounded} if, for each $\ell \in \ZZ$, the set of periods of contractible Reeb orbits of $\partial K$ of Conley-Zehnder index $\ell$ is bounded.

As in the case of Floer cohomology, we can introduce action filtrations of $SH_M(K)$ and $SH^{S^1}_M(K)$. Let $H_1, H_2 \in  \mathcal{H}_K^{\text{Cont}}$ satisfying $H_1 \leq H_2$. Then the continuation map connecting two Floer complexes $CF_w(H_1)$ and $CF_w(H_2)$ increases the action and thus the family $\left\{ CF_w^{>L}(H) \right\}_{H \in  \mathcal{H}_K^{\text{Cont}}}$ still forms a direct system together with continuation maps for each $L \in \RR$. We can define
\begin{align*}
    SH^{> L,*}_M(K) = H^* \left( \widehat{\varinjlim_{H \in \mathcal{H}_K^{\text{Cont}}}} CF_w^{> L}(H) \right).
\end{align*}
We can consider the \textit{negative relative symplectic cohomology} $SH^{-}_M(K)$. Roughly speaking, it is generated by the nonconstant Hamiltonian orbits of $H \in \mathcal{H}_K^{\text{Cont}}$. We denote the action filtration of $SH^{-}_M(K)$ by $SH^{-,>L}_M(K)$. The action filtration $SH^{S^1,>L}_M(K)$ of $SH^{S^1}_M(K)$ can be defined in an analogous way. Also, we can define the \textit{negative $S^1$-equivariant relative symplectic cohomology} $SH^{S^1,-}_M(K)$ and its action filtration $SH^{S^1,-,>L}_M(K)$. Detailed exposition can be found in \cite{a}.

Keys to define the relative symplectic (co)homology capacity and the relative Gutt-Hutchings capacity are the following exact triangles combining various cohomologies.
\begin{theorem}[\cite{a}]\label{ing} We have the following exact triangles
\begin{align}\label{ex}
    \includegraphics[scale=1.2]{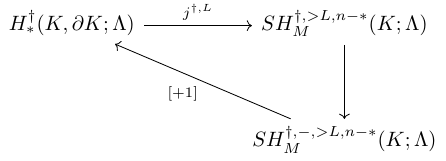}
\end{align}
where $[+1]$ denotes the degree 1 map with respect to $*$ and $\dagger = \emptyset$ or $S^1$.
     
\end{theorem}
We denote the horizontal map in \eqref{ex} for $\dagger = \emptyset$ and $S^1$ by
\begin{align*}
    j^{\emptyset,L} = j^L : H(K,\partial K; \Lambda) \lr SH^{>L}_M(K ;\Lambda)
\end{align*}
and
\begin{align*}
    j^{S^1,L} : H^{S^1}(K,\partial K; \Lambda) \lr SH^{S^1,>L}_M(K ;\Lambda),
\end{align*}
respectively. Note that
\begin{align*}
  H_*^{S^1}(K,\partial K; \Lambda) &=  H_*(BS^1;\Lambda) \otimes H_*(K,\partial K; \Lambda)\\& = H_*(\CC P^{\infty};\Lambda) \otimes H_*(K,\partial K; \Lambda) \\&= \Lambda[u]  \otimes H_*(K,\partial K; \Lambda)
\end{align*}
since the $S^1$-action on the pair $(K,\partial K)$ is trivial. The \textbf{relative symplectic (co)homology capacity} $c^{SH}(M,K)$ of $K$ inside $M$ is defined by
        \begin{align*}
            c^{SH}(M,K) = - \sup \left\{ L<0 \bigmid j^L\left([K,\partial K]\right) = 0 \right\}
        \end{align*}
        where $[K, \partial K]$ is the fundamental class of $H_{2n}(K,\partial K; \Lambda)$. And the \textbf{first relative Gutt-Hutchings capacity} $c_1^{GH}(M,K)$ of $K$ inside $M$ is defined by
        \begin{align*}
            c_1^{GH}(M,K) = - \sup \left\{ L<0 \bigmid j^{S^1,L}(1\otimes\left[K,\partial K]\right) = 0 \right\}
        \end{align*}
        where $1 \in H(BS^1;\Lambda)$. It is proved in \cite{a} that
        \begin{align}\label{inq}
            c_1^{GH}(M,K) \leq c^{SH}(M,K).
        \end{align}
        Furthermore, if the contact manifold $(\partial K, \alpha)$ is dynamically convex, then we can prove the other inequality $$c^{SH}(M,K) \leq c_1^{GH}(M,K)$$ and hence
        $c_1^{GH}(M,K) = c^{SH}(M,K)$.
        \begin{remark}
            The definition of $k$-th relative Gutt-Hutchings capacity for each $k=1,2,3,\cdots$ can be found in \cite{a}.
        \end{remark}


        \medskip
\section{Comparisons}
\subsection{Comparison of $c^S(M,K)$ with $c^{SH}(M,K)$} To compare $c^S(M,K)$ with $c^{SH}(M,K)$, we need several ingredients. The following lemma and its corollary are our first ingredients.

\begin{lemma}\label{onemin}
       Let $(M, \omega)$ be a closed symplectic manifold and $K\subset M$ be a compact domain with contact type boundary. Then there exists a Morse function $f : M \to \RR $ that satisfies the following.
       \begin{enumerate}[label=(\alph*)]
           \item $f$ is negative on $K$.
           \item $f$ has only one critical point corresponding to the minimum of $f$ on $K$.
           \item There are no Morse trajectories of $f$ connecting a critical point of $f$ outside of $K$ to a critical point of $f$ inside of $K$.
       \end{enumerate}
                 
   \end{lemma}
    \begin{proof}
        Recall that every smooth manifold admits a CW structure with a $k$-cell for each critical point of index $k$ of a Morse function defined on it. Using this fact, we first construct a Morse function on $K$ which has only one critical point corresponding to the minimum. To do this, let $g : K \to \RR$ be a Morse function. Then the critical points of $g$ corresponding to local minima of $g$ form a 0-skeleton of $K$. Let $T$ be a maximal tree, a contractible subset of 1-skeleton that reaches all the vertices, of $K$. Since $T$ is contractible by its definition, the projection map $M \lr M/T$ is a homotopy equivalent. Since $M/T$ has only one 0-cell, we can homotope $g$ so that it has only one critical point of index 0. We still denote this transformed function by $g$. We may assume that $g$ is negative. Now it remains to extend the function $g: K \to \RR$ to $f : M\to \RR$ to satisfy (c). Let $\partial K \times [0, \delta]$ be a small collar neighborhood of $\partial K$. We extend $g$ as follows.
        \begin{itemize}
            \item $f(p, r) = h_1(e^r)$ on $\partial K \times [0, \frac{\delta}{2}]$ for some strictly increasing and convex function $h_1$.\\
            \item $f(p, r) = h_2(e^r)$ on $\partial K \times [\frac{\delta}{2}, \delta]$ for some strictly increasing and concave function $h_2$.\\
            \item $f$ is a Morse function $C^2$-close to a constant function on $M - (K \cup (\partial K \times [0, \delta]))$.\\
        \end{itemize}
        Note that there are no critical points of $f$ on $\partial K \times [0,\delta]$ by its construction. For critical points $x,y \in \text{Crit}(f)$ and a Morse trajectory $\gamma:\RR \lr M$ satisfying
    \begin{align*}
        \lim_{s\to -\infty} \gamma(s) = x\,\,\text{and}\,\, \lim_{s\to \infty} \gamma(s) = y,
    \end{align*}
    we have 
    \begin{align*}
        f(y) - f(x) &= \int_{-\infty}^{\infty} \frac{d}{ds}f(\gamma(s)) ds\\& = \int_{-\infty}^{\infty} g(\nabla_gf(\gamma(s)), \gamma'(s)) ds\\& =  \int_{-\infty}^{\infty} g(\nabla_gf(\gamma(s)), \nabla_gf(\gamma(s))) ds\\&\geq 0
    \end{align*}
    where $g$ is a generic metric on $M$, and hence the critical value increases along any Morse trajectory. Since the critical values that appear outside of $K$ are greater than the critical values corresponding to the critical points in $K$, the third condition (c) is satisfied.\\
    \end{proof}
  
         The point of Lemma \ref{onemin} is the following corollary.
  
   \begin{corollary}\label{pinch}
        Let $(M, \omega)$ be a closed symplectic manifold and $K\subset M$ be a Liouville domain. Then there exists a map 
        \begin{align}\label{mtok}
            \rho :H^*(M ;\Lambda) \lr H_{2n-*}(K, \partial K;\Lambda).
        \end{align}
        
    \end{corollary}
\begin{proof}
    Let $f : M \to \RR$ be the function constructed in Lemma \ref{onemin}. Moreover, we may assume that $f$ is $C^2$-small on $K$. Let $$K' = M - (K \cup (\partial K\times [0,\delta))).$$ Since there are no Morse trajectories of $f$ from a critical point outside of $K$ to a critical point inside of $K$, the Morse complex $CM(f|_{K'})$ of the restriction $f|_{K'}$ is a subcomplex of $CM(f)$. Then there exists a short exact sequence of chain complexes
    \begin{align}\label{ses}
        0\lr CM(f|_{K'}) \lr CM(f) \lr CM(f)/CM(f|_{K'}) \lr0.
    \end{align}
    This short exact sequence \eqref{ses} induces a long exact sequence
    \begin{align}\label{les}
        \includegraphics[scale=1.2]{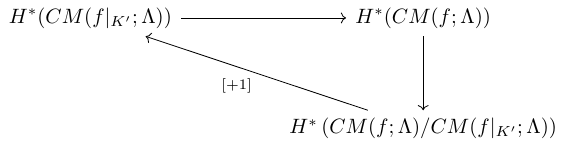}.
    \end{align}
    Note that $CM(f)/CM(f|_{K'})$ is generated by the critical points of $f$ on $K$. As $K$ is a Liouville domain, every Floer trajectory connecting Hamiltonian orbits on $K$ has its image entirely in $K$. See Lemma 2.2 of \cite{co}. As a consequence of $f$ being $C^2$-small on $K$, we can reformulate this statement as: every Morse trajectory connecting critical points in $K$ has its image inside $K$. Therefore,
    \begin{align}\label{noes}
        CM(f)/CM(f|_{K'}) \cong CM(f|_K).
    \end{align}
    Obviously, $$H^*(CM(f;\Lambda)) \cong H^*(M;\Lambda).$$
    Since the gradient of $f$ points outward along $\partial K$, we have
    \begin{align*}
        H^*(CM(f;\Lambda)/CM(f|_{K'};\Lambda)) &\cong H^*(CM(f|_K;\Lambda))&&\text{By \eqref{noes}}\\
        &\cong H_{2n-*}(K, \partial K;\Lambda) &&\text{By (c) of Remark \ref{morse}}
    \end{align*}
    The vertical map in \eqref{les} is the desired map $\rho : H(M ;\Lambda) \lr H(K, \partial K;\Lambda)$.
 
\end{proof}

\begin{remark}
    The map $\rho : H^*(M ;\Lambda) \lr H_{2n-*}(K, \partial K;\Lambda) $ can be derived as follows: We view $H^*(M ;\Lambda)$ as the de Rham cohomology of $M$. Then the restriction of differential forms to $K$ yields a map
     \begin{align*}
        H^*(M;\Lambda) \lr H^*(K;\Lambda),\,\,\mu \mapsto \mu|_{K}.
    \end{align*}
    Since there exists an isomorphism
    \begin{align*}
        H^*(K;\Lambda) \cong H_{2n-*}(K,\partial K;\Lambda)
    \end{align*}
    by Lefschetz duality, we also get \eqref{mtok}. 
\end{remark}

Note that the map \eqref{mtok} maps the unit element $1_M \in QH^0(M;\Lambda)$ to the fundamental class $[K,\partial K] \in H_{2n}(K,\partial K;\Lambda)$ because it maps the the critical point corresponding to the minimum of $f$ to the critical point corresponding to the maximum of $-f$. One more input that we need is an analogue of the following isomorphism constructed in \cite{i} by Irie.

\begin{theorem}[\cite{i}, Lemma 2.5]\label{i}
    Let $K$ be a Liouville domain. For each $L<0$ such that $-L$ is not a period of any Reeb orbits on $\partial K$, there exists a Hamiltonian function $H : S^1 \times \widehat{K} \to \RR$ such that 
    \begin{align}\label{ii}
           HF(H ; \Lambda) \cong SH^{>L}(K;\Lambda).
    \end{align}
    Moreover, this isomorphism \eqref{ii} commutes with continuation maps.
\end{theorem}

In the proof of the next Lemma, we need the following algebraic fact.
\begin{theorem}[Mittag-Leffler]\label{ml} 
    Let $R$ be a ring with unit element and let $(I, \leq)$ be a directed index set. Let $\{A_i, \,r_{ij}^A : A_i \lr A_j \}_{i \in I}$, $\{B_i, \,r_{ij}^B : B_i \lr B_j \}_{i \in I}$ and $\{C_i, \,r_{ij}^C : C_i \lr C_j \}_{i \in I}$ be inverse systems of $R$-modules. Suppose that for all $i \in I$, we have a short exact sequence
    \begin{align*}
        0 \lr A_i \lr B_i \lr C_i \lr 0. 
    \end{align*}
    Suppose further that for each $i \in I$, there exists $j \geq i$ such that $r^A_{ki}(A_k) = r^A_{ji}(A_j)$ for $k \geq j$. If the index set $I$ is countable, then
    \begin{align*}
        0 \lr \varprojlim_{i \in I} A_i \lr \varprojlim_{i \in I} B_i \lr \varprojlim_{i \in I} C_i \lr 0
    \end{align*}
  is exact.
\end{theorem}
Before we state and prove an analogous statement of Theorem \ref{i}, let us introduce some terminology: Let $H\in \mathcal{H}_K^{\text{Cont}}$ and $x \in \mathcal{P}(H)$. We call $x$ a \textit{lower orbit} if it is a constant orbit inside $K$ or it is a nonconstant orbit coming from the convex part of $H$. We call $x$ an \textit{upper orbit} if it is a constant orbit outside $K$ or it is a nonconstant orbit coming from the concave part of $H$. 

\begin{lemma}\label{iri}
    Let $(M,\omega)$ be a closed symplectically aspherical symplectic manifold and $K\subset M$ be a Liouville domain with index-bounded boundary. For each $L<0$ such that $-L$ is not an element of $\text{Spec}(\partial K,\alpha)$, there exists $H \in \mathcal{H}_K^{\text{Cont}}$ and a canonical isomorphism 
    \begin{align*}
        \eta :  HF^{\leq 0}(H;\Lambda) \xlongrightarrow{\cong}  SH^{>L}_M(K;\Lambda).
    \end{align*}
\end{lemma}
\begin{proof}
Choose $\delta>0$ small enough so we can consider a neighborhood $\partial K \times [0,\delta]$ of $\partial K$ in $M$. Define $H :S^1 \times M \to \RR$ to satisfy
\begin{itemize}
    \item $H$ is negative and $C^2$-small on $S^1 \times K$,\\
    \item $H(t,p,\rho) $ is $C^2$-close to $ h_1(e^{\rho})$ on $S^1 \times (\partial K \times [0, \frac{1}{3}\delta])$ for some strictly convex and increasing function $h_1$,\\
    \item $H(t,p,\rho) =  -L e^{\rho} +L'$ on $S^1 \times (\partial K \times [\frac{1}{3}\delta, \frac{2}{3}\delta])$ for some $L' \in \RR$,\\
    \item $H(t,p,\rho) $ is $C^2$-close to $ h_2(e^{\rho})$ on $S^1 \times (\partial K \times [\frac{2}{3}\delta, \delta])$ for some strictly concave and increasing function $h_2$, and\\
    \item $H$ is $C^2$-close to a constant function on $S^1 \times (M - (K \cup (\partial K \times [0,\delta])))$. \\
\end{itemize}
Then obviously $H \in \mathcal{H}_K^{\text{Cont}}$. Given that $\delta >0$ is sufficiently small, we can assure that the actions of lower orbits of $H$ (first and second bullet points) are greater than $L$ and actions of upper orbits of $H$ (fourth and fifth bullet points) are greater than 0 considering the formula given by \eqref{for}. For $G \in \mathcal{H}_K^{\text{Cont}}$, let
\begin{align*}
 CF_{w,\uparrow}(G) = \left\{x\in CF_w(G) \bigmid x\,\,\text{is an upper orbit of}\,\, G \right\}.   
\end{align*}
Because the Floer differential of $CF_w(G)$ increases the actions, there are no Floer trajectories from upper orbits to lower orbits and hence the subset $CF_{w,\uparrow}(G)$ is actually a subcomplex of $CF_{w}(G)$. Moreover, $CF_{w,\uparrow}(G)$ is a subcomplex of $CF^{>L}_{w}(G)$ because every upper orbit of $G$ has action greater than 0 and thus greater than $L$. Let
\begin{align*}
    CF^{>L}_{w,\downarrow}(G) = CF^{>L}_{w}(G)/CF_{w,\uparrow}(G).
\end{align*}
Note that it is proved in \cite{co} (Lemma 2.2) that every Floer trajectory connecting lower orbits lies inside $K \cup ( \partial K \times [0, \frac{1}{3} \delta])$. The short exact sequence 
\begin{align*}
    0 \lr CF_{w,\uparrow}(G)\lr CF^{>L}_w(G)\lr CF^{>L}_{w,\downarrow}(G) \lr 0
\end{align*}
induces 
\begin{align*}
    0 \lr \varinjlim_{G \in \mathcal{H}_K^{\text{Cont}}} CF_{w,\uparrow}(G)\lr \varinjlim_{G \in \mathcal{H}_K^{\text{Cont}}} CF^{>L}_w(G)\lr \varinjlim_{G \in \mathcal{H}_K^{\text{Cont}}} CF^{>L}_{w,\downarrow}(G) \lr 0.
\end{align*}
because the direct limit is an exact functor. Note that $CF^{>L}_{w,\downarrow}(G)$ is a free module over $\Lambda_{\geq 0}$ generated by lower orbits of $G$ with action greater than $L$ and hence it is a flat module over $\Lambda_{\geq0}$. Since the direct limit of flat modules is known to be flat, we have
\begin{align*}
    \text{Tor}_1^{\Lambda_{\geq 0}}\left(\varinjlim_{G \in \mathcal{H}_K^{\text{Cont}}}CF^{>L}_{w}(G), \Lambda_{\geq 0}/ \Lambda_{\geq r}\right) = 0
\end{align*}
and hence
\begin{align*}
    0 \lr \varinjlim_{G \in \mathcal{H}_K^{\text{Cont}}} CF_{w,\uparrow}(G) \ton \Lambda_{\geq0}/\Lambda_{\geq r}&\lr \varinjlim_{G \in \mathcal{H}_K^{\text{Cont}}}CF^{>L}_w(G)\ton \Lambda_{\geq0}/\Lambda_{\geq r}\\&\lr \varinjlim_{G \in \mathcal{H}_K^{\text{Cont}}}CF^{>L}_{w,\downarrow}(G)\ton \Lambda_{\geq0}/\Lambda_{\geq r} \lr 0
\end{align*}
is still exact for each $r\geq0$. For $r' > r$, the projection
\begin{align}\label{mlex}
    \Lambda_{\geq 0}/ \Lambda_{\geq r'} \to \Lambda_{\geq 0}/ \Lambda_{\geq r}
\end{align}
is surjective. For any $\Lambda_{\geq 0}$-module $A$, the map induced by \eqref{mlex} $$A \ton \Lambda_{\geq 0}/ \Lambda_{\geq r'} \to A \ton \Lambda_{\geq 0}/ \Lambda_{\geq r}$$ induced by \eqref{mlex} is also surjective due to the right exactness of tensor product. Then by the Mittag-Leffler theorem for the inverse limit, the following sequence
\begin{align*}
    0 \lr \varprojlim_{r\to0}\varinjlim_{G \in \mathcal{H}_K^{\text{Cont}}} CF_{w,\uparrow}(G) \ton \Lambda_{\geq0}/\Lambda_{\geq r}&\lr\varprojlim_{r\to0} \varinjlim_{G \in \mathcal{H}_K^{\text{Cont}}}CF^{>L}_w(G)\ton \Lambda_{\geq0}/\Lambda_{\geq r}\\&\lr \varprojlim_{r\to0} \varinjlim_{G \in \mathcal{H}_K^{\text{Cont}}}CF^{>L}_{w,\downarrow}(G)\ton \Lambda_{\geq0}/\Lambda_{\geq r} \lr 0
\end{align*}
is also exact. Note that the index set $\{r>0\}$ is not countable so the Mittag-Leffler theorem does not apply directly. But we have a countable cofinal index subset $$\left\{\frac{1}{n} \bigmid n=1,2,3,\cdots\right\}$$ of $\{r>0\}$ which satisfies the assumption of Theorem \ref{ml}. By the definition of the completion, we have the following short exact sequence

\begin{align}\label{lmc}
    0 \lr \widehat{\varinjlim_{G \in \mathcal{H}_K^{\text{Cont}}}} CF_{w,\uparrow}(G)\lr \widehat{\varinjlim_{G \in \mathcal{H}_K^{\text{Cont}}}} CF^{>L}_w(G)\lr \widehat{\varinjlim_{G \in \mathcal{H}_K^{\text{Cont}}}} CF^{>L}_{w,\downarrow}(G) \lr 0.
\end{align}
We can make a couple of modifications of \eqref{lmc}. First, it is proved in \cite{dgpz} (Lemma 5.3) that the leftmost chain complex of \eqref{lmc} is zero, so the short exact sequence \eqref{lmc} becomes an isomorphism of chain complexes
\begin{align*}
    \widehat{\varinjlim_{G \in \mathcal{H}_K^{\text{Cont}}}} CF^{>L}_w(G) \cong \widehat{\varinjlim_{G \in \mathcal{H}_K^{\text{Cont}}}} CF^{>L}_{w,\downarrow}(G).
\end{align*}
One more alteration of \eqref{lmc} is that we can get rid of the hat symbol over $\displaystyle\varinjlim_{G \in \mathcal{H}_K^{\text{Cont}}} CF^{>L}_{w,\downarrow}(G)$ because $\displaystyle\varinjlim_{G \in \mathcal{H}_K^{\text{Cont}}} CF^{>L}_{w,\downarrow}(G)$ is complete, which is also proved in \cite{dgpz} (Lemma 5.4) and hence
\begin{align*}
    \widehat{\varinjlim_{G \in \mathcal{H}_K^{\text{Cont}}}} CF^{>L}_{w,\downarrow}(G) \cong \varinjlim_{G \in \mathcal{H}_K^{\text{Cont}}} CF^{>L}_{w,\downarrow}(G).
\end{align*}
We claim that 
\begin{align}\label{irilem}
     H \left( CF_{w, \downarrow}(H) \right) \cong  H \left( \varinjlim_{\substack{G\in \mathcal{H}_K^{\text{Cont}}}} CF^{>L}_{w,\downarrow}(G) \right).
\end{align}
A similar isomorphism is established in \cite{i}, so the proof of the claim \eqref{irilem} is inspired by the proof of Lemma 2.5 of \cite{i}. For $G \in \mathcal{H}_K^{\text{Cont}} $, denote the slope of the linear part of $G$ by $s_G$. Let
\begin{align*}
    \mathcal{H}_K^{\text{Cont},\leq -L} =\left\{ G \in \mathcal{H}_K^{\text{Cont}} \bigmid s_G \leq -L \right\}.
\end{align*}
Since $H \in \mathcal{H}_K^{\text{Cont},\leq -L}$, there exists a morphism 
\begin{align}\label{1iso}
    H\left( CF_{w,\downarrow}(H) \right) \lr H\left( \varinjlim_{G \in \mathcal{H}_K^{\text{Cont},\leq -L}}  CF_{w,\downarrow}(G)\right).
\end{align}
Let $\{G_k\}$ be a cofinal sequence of $\mathcal{H}_K^{\text{Cont},\leq -L}$. Note that we may assume that $G_k$ is $C^2$-small on $S^1 \times K$ for every $k=1,2,3,\cdots$.
We can construct the following commutative diagram
\begin{align}\label{tri}
    \includegraphics[scale=1.2]{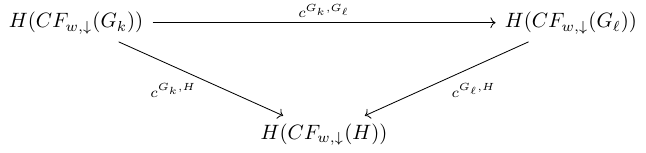}
\end{align}
where every map in \eqref{tri} is induced by the continuation map. By the universal property of the direct limit, we have a morphism
\begin{align*}
   H \left( \varinjlim_{G \in \mathcal{H}_K^{\text{Cont},\leq -L}} CF_{w,\downarrow}(G)\right) \cong H \left( \varinjlim_{k\to \infty} CF_{w,\downarrow}(G_k)\right) \lr  H\left( CF_{w,\downarrow}(H) \right),
\end{align*}
which is the inverse of the map \eqref{1iso} and hence
\begin{align}\label{s1}
    H\left( CF_{w,\downarrow}(H) \right) \cong H \left( \varinjlim_{G \in \mathcal{H}_K^{\text{Cont},\leq -L}} CF_{w,\downarrow}(G)\right).
\end{align}
Additionally, we can easily deduce that $$CF_{w,\downarrow}(G_k) = CF_{w,\downarrow}^{>L}(G_k)$$ for each $k=1,2,3,\cdots$ because $G_k$ is $C^2$-small on $S^1 \times K$ and every nonconstant orbit of $G_k$ has action greater than $L$ due to the restriction on $s_{G_k}$. So,
\begin{align}\label{s2}
    H \left( \varinjlim_{G \in \mathcal{H}_K^{\text{Cont},\leq -L}} CF_{w,\downarrow}(G)\right) = H\left(\varinjlim_{G \in \mathcal{H}_K^{\text{Cont},\leq -L}} CF_{w,\downarrow}^{>L}(G)\right).
\end{align}
Let $G \in \mathcal{H}_K^{\text{Cont}}$ with $s_G > -L$ and $G$ being $C^2$-small on $S^1 \times K$. We can choose a small number $\eta >0$ such that $G$ is given by a strictly convex and increasing function $g_1$ on $\partial K\times[0,\frac{1}{3}\eta]$, linear on $\partial K\times[\frac{1}{3}\eta,\frac{2}{3}\eta]$ and a strictly concave function on $K\times[\frac{2}{3}\eta,\eta]$. We can find $\rho_0 \in (0,\frac{1}{3}\eta)$ such that $g_1'(e^{\rho_0}) = -L$. 
Define $\widetilde{G} \in \mathcal{H}_K^{\text{Cont},\leq -L}$ by
\begin{itemize}
    \item $\widetilde{G} = G$ on $K \cup (\partial K \times [0,\rho_0])$,\\
    \item $\widetilde{G}(t,p,\rho) = -Le^{\rho}+L''$ on $S^1 \times (\partial K \times [\rho_0, 2 \rho_0])$ for some $L''\in\RR$,\\
    \item $\widetilde{G}(t,p,\rho) $ is $C^2$-close to $ g_2(e^\rho)$ for some strictly concave and increasing function $g_2$ on $\partial K \times [2\rho_0,3\rho_0]$, and\\
    \item $\widetilde{G}$ is $C^2$-close to a constant function on $S^1 \times (M - (K \cup (\partial K \times [0,3\rho_0])))$. \\
\end{itemize}
Then we have an isomorphism $$H(CF_{w,\downarrow}^{>L}(G)) \cong H(CF_{w,\downarrow}^{>L}(\widetilde{G}))$$ and therefore
\begin{align}\label{s3}
   H\left( \varinjlim_{G \in \mathcal{H}_K^{\text{Cont},\leq -L}} CF_{w,\downarrow}^{>L}(G)\right) \cong H\left( \varinjlim_{G \in \mathcal{H}_K^{\text{Cont}}} CF_{w,\downarrow}^{>L}(G)\right).
\end{align}
Putting \eqref{s1}, \eqref{s2} and \eqref{s3} together, we prove the claim \eqref{irilem}. From the fact that every upper orbit of $H$ has action greater than 0, we have
\begin{align}\label{remove}
    H\left( CF_{w,\downarrow}(H)  \ton \Lambda \right) \cong HF^{\leq 0}(H ;\Lambda).
\end{align}
Note that we can remove subscript $_w$ in \eqref{remove} by (a) of Remark \ref{morse}.
The core of the preceding discussion is captured by the following chain of isomorphisms.
\begin{align*}
    SH^{>L}_M(K;\Lambda) &= H\left( \widehat{\varinjlim_{G \in \mathcal{H}_K^{\text{Cont}}}} CF^{>L}_w(G) \ton \Lambda \right) &&\text{Definition}\\
    &\cong H\left( \widehat{\varinjlim_{G \in \mathcal{H}_K^{\text{Cont}}}} CF^{>L}_{w,\downarrow}(G) \ton \Lambda \right) && \text{Ignore upper orbits}\\
   &\cong H\left(  \varinjlim_{G \in \mathcal{H}_K^{\text{Cont}}} CF^{>L}_{w,\downarrow}(G) \ton \Lambda\right)&&\varinjlim_{G \in \mathcal{H}_K^{\text{Cont}}} CF^{>L}_{w,\downarrow}(G)\,\,\text{is complete}\\
&\cong H\left( CF_{w,\downarrow}(H) \ton \Lambda  \right) &&\text{By \eqref{irilem}}\\
&\cong HF^{\leq 0}\left( H; \Lambda  \right)&&\text{By}\,\, \eqref{remove}
   \end{align*}
\end{proof}
Note that for a nondegenerate Hamiltonian function $H : S^1 \times K \to \RR$ and $a,b \in \RR$ with $a\leq b$, there exists a map
\begin{align}\label{view}
    \pi^{H, a,b} : HF^{\leq b}(H) \lr HF^{\leq a}(H)
\end{align}
induced by the map
\begin{align*}
    CF(H) / CF^{>b}(H) \lr CF(H) / CF^{>a}(H). 
\end{align*}
Of course, for $a\leq b$, we have
\begin{align}\label{view2}
    \pi^{H,\leq a} = \pi^{H,a,b} \circ \pi^{H,\leq b}.
\end{align}
Thus far we have gathered all the ingredients that we need to prove the following theorem. The idea of the proof is inspired by \cite{bk}.
\begin{theorem}\label{mcomparison}
    Let $(M,\omega)$ be a closed symplectically aspherical symplectic manifold and $K\subset M$ be a Liouville domain with index-bounded boundary. Then
    \begin{align*}
        c^S(M,K) \leq c^{SH}(M,K).
    \end{align*}
\end{theorem}
\begin{proof}\
We only need to consider the case when $c^{SH}(M,K)$ is finite. Choose any $a > c^{SH}(M,K)$ such that $a \notin \text{Spec}(\partial K, \alpha)$. This is always possible because the set $\text{Spec}(\partial K, \alpha)$ is known to be a nowhere dense closed subset of $\RR$. Let $H$ be a nonpositive Hamiltonian function supported in $S^1\times (K-\partial K)$. It suffices to show that
\begin{align}\label{goal}
    -\rho(1_M;H) \leq a.
\end{align}
we can find a sufficiently small $\delta >0 $ such that $H = 0$ on $S^1 \times (\partial K \times (-\delta,0])$. By the Lipshcitz property \eqref{lp} of spectral invariants, we may assume that $H$ is nondegenerate and negative and $C^2$-small on $S^1 \times (\partial K \times (-\delta,0])$. Define $\widetilde{H} : S^1 \times M \to \RR$, in a similar way that we did in Lemma \ref{iri}, to satisfy 
\begin{itemize}
    \item $\widetilde{H} = H$ on $S^1 \times K$,\\
    \item $\widetilde{H}(t,p,\rho) $ is $C^2$-close to $h_1(e^{\rho})$ on $S^1 \times (\partial K \times [0, \frac{1}{3}\delta])$ for some strictly convex and increasing function $h_1$,\\
    \item $\widetilde{H}(t,p,\rho) =  a e^{\rho} +a'$ on $S^1 \times (\partial K \times [\frac{1}{3}\delta, \frac{2}{3}\delta])$ for some $a' \in \RR$,\\
    \item $\widetilde{H}(t,p,\rho) $ is $C^2$-close to $ h_2(e^{\rho})$ on $S^1 \times (\partial K \times [\frac{2}{3}\delta, \delta])$ for some strictly concave and increasing function $h_2$, and\\
    \item $\widetilde{H}$ is $C^2$-close to a constant function on $S^1 \times (M -( K \cup (\partial K \times [0,\delta])))$. \\
\end{itemize}
As we saw in the proof of Lemma \ref{iri}, there exists an isomorphism
\begin{align}\label{ii2}
   \eta: HF^{\leq 0}(\widetilde{H};\Lambda) \cong SH^{>-a}_M(K;\Lambda).
\end{align}
Now let's consider the following commutative diagram.
\begin{align}\label{first}
    \includegraphics[scale=1.2]{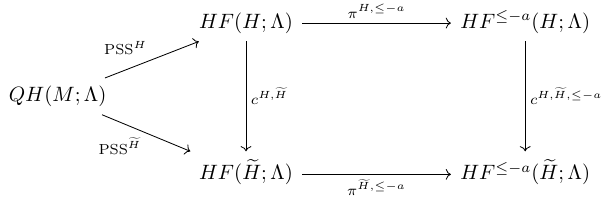}
\end{align}
Note that the second vertical map 
\begin{align*}
    c^{H,\widetilde{H},\leq -a} : HF^{\leq -a}(H;\Lambda) \lr HF^{\leq -a} (\widetilde{H};\Lambda)
\end{align*}
in \eqref{first} is induced by the continuation map $c^{H,\widetilde{H}}$ and it is an isomorphism because every Hamiltonian orbit of $H$ with action less than or equal to $-a$ occurs where $\{H = \widetilde{H}\}$ and every Hamiltonian orbit of $\widetilde{H}$ with action less than or equal to $-a$ also occurs on the same region. Hence, to see \eqref{goal}, we will show that 
\begin{align}\label{goal2}
    \pi^{\widetilde{H},\leq-a} \circ \text{PSS}^{\widetilde{H}}(1_M) = 0.
\end{align}
But in view of the map \eqref{view} and the relation \eqref{view2}, our goal \eqref{goal2} can be reformulated as 
\begin{align}\label{goal3}
    \pi^{\widetilde{H},\leq 0} \circ \text{PSS}^{\widetilde{H}}(1_M) = 0.
\end{align}
To achieve our newly set goal \eqref{goal3}, commutative diagram below is helpful.
\begin{align}\label{5gon}
    \includegraphics[scale=1.2]{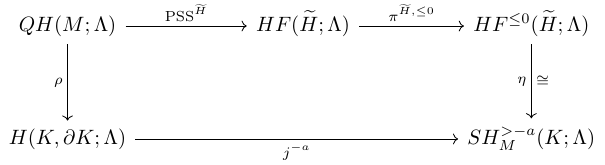}
\end{align}
The first vertical map of \eqref{5gon} is the map given by \eqref{mtok} and the second vertical map is the isomorphism $\eta$ given by \eqref{ii2}. Lastly, the map in the second row is given in Theorem \ref{ing}. Since $a > c^{SH}(M,K)$, we have $j^{-a} ([K, \partial K]) = 0$ and therefore
\begin{align*}
    \eta \left( \pi^{\widetilde{H},\leq 0} \circ \text{PSS}^{\widetilde{H}}(1_M) \right)= j^{-a}\circ\rho(1_M) = j^{-a}([K, \partial K]) = 0.
\end{align*}
Since $\eta$ is an isomorphism, we have $\pi^{\widetilde{H},\leq 0} \circ \text{PSS}^{\widetilde{H}}(1_M) = 0$. The essence of this proof is summarized in the commutative diagram below.
\begin{align*}
    \includegraphics[scale=1.2]{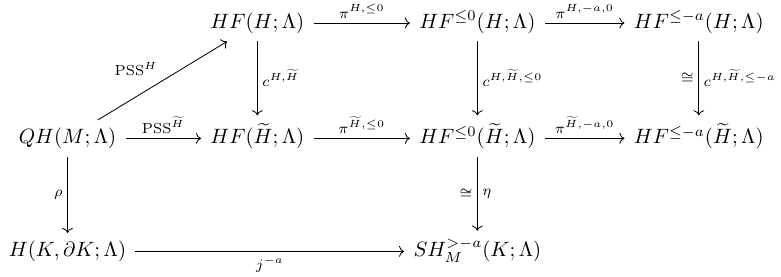}
\end{align*}

\end{proof}
\medskip
\subsection{Comparison of $c_1^{GH}(M,K)$ with $c^{SH}(M,K)$}
Let us remind the Gysin-type exact triangle regarding relative symplectic cohomologies constructed in \cite{a}.
\begin{theorem}[\cite{a}]\label{rbo}
    Let $(M,\omega)$ be a closed symplectically aspherical symplectic manifold and $K\subset M$ be a Liouville domain with index-bounded boundary. Then we have the following exact triangle
    \begin{align}\label{gysin}
        \includegraphics[scale=1.2]{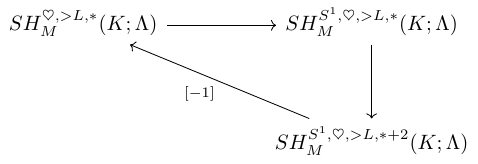}
    \end{align}
    where [-1] denotes the degree shift by $-1$ with respect to $*$ and $\heartsuit= \emptyset$ or $-$.
\end{theorem}
We denote the horizontal map of the exact triangle \eqref{gysin} by
\begin{align*}
    \xi^L : SH^{>L}_M(K;\Lambda) \lr SH^{S^1, >L}_M(K;\Lambda)
\end{align*}
if $\heartsuit = \emptyset$ and denote it by
\begin{align*}
    \zeta^L : SH^{-,>L}_M(K;\Lambda) \lr SH^{S^1,-, >L}_M(K;\Lambda)
\end{align*}
if $\heartsuit= -$. Note that the maps $\xi^L$ and $\zeta^L$ are induced by $x\mapsto 1\otimes x$ on the chain level.

\begin{lemma}\label{inj}
    Let $(M,\omega)$ be a $2n$-dimensional closed symplectically aspherical symplectic manifold and $K\subset M$ be a Liouville domain with index-bounded boundary. Suppose that every contractible Reeb orbit $\gamma$ of $(\partial K, \alpha)$ satisfies $\mu_{\text{CZ}} (\gamma) \geq n$.
    \begin{enumerate}[label=(\alph*)]
        \item The map 
        \begin{align*}
            \zeta^{L,-n-1} : SH^{-,>L,-n-1}_M(K;\Lambda) \lr SH^{S^1,-, >L,-n-1}_M(K;\Lambda)
        \end{align*}
        is surjective.
        \item If, furthermore, there exists a Morse function on $K$ which does not admit critical points of odd Morse index, then the map
        \begin{align*}
            \xi^{L,-n} : SH^{>L,-n}_M(K;\Lambda) \lr SH^{S^1, >L,-n}_M(K;\Lambda)
        \end{align*}
        is injective.
        
    \end{enumerate}
    
\end{lemma}
\begin{proof}
    (a) Magnifying the part that we need from the exact triangle \eqref{gysin}, we have
    \begin{align*}
    \cdots\lr SH^{-,>L,-n-1}_M(K;\Lambda) \xlongrightarrow{\zeta^{L,-n-1}} SH^{S^1,-,>L,-n-1}_M(K;\Lambda) \lr SH^{-,>L,-n+1}_M(K;\Lambda) \lr \cdots. 
    \end{align*}
    To prove that $\zeta^{L,-n-1}$ is surjective, we shall show that
    \begin{align}\label{sur}
        SH^{-,>L,-n+1}_M(K;\Lambda) = 0.
    \end{align}
    Let $H \in \mathcal{H}_K^{\text{Cont}}$ and let $u^{\ell} \otimes x \in \Lambda_{\geq 0}[u] \otimes CF_w(H)$ be a generator of $ CF_w^{S^1,-,>L,-n+1}(H)$. Note that $x$ is a nonconstant Hamiltonian orbit $H$. Then by the grading convention given in \eqref{grcon}, we have
    \begin{align*}
        | u^{\ell} \otimes x| = -2\ell + \mu_{\text{CZ}}(x) = - n +1
    \end{align*}
    and hence
    \begin{align}\label{cont1}
        \mu_{\text{CZ}}(x) = - n + 2 \ell +1.
    \end{align}
    Since the Hamiltonian orbit $x$ travels in the opposite direction of Reeb orbits on $(\partial K, \alpha)$ as explained in \eqref{hamorbit}, the Conley-Zehnder index of $x$ satisfies
    \begin{align}\label{cont2}
        \mu_{\text{CZ}}(x) \leq -n
    \end{align}
    by the assumption. Combining \eqref{cont1} and \eqref{cont2}, we obtain the inequality
    \begin{align*}
        2 \ell + 1 \leq 0,
    \end{align*}
    which is impossible. The preceding discussion implies that
    \begin{align*}
        CF_w^{S^1,-,>L,-n+1}(H) = 0
    \end{align*}
    for every $H \in \mathcal{H}_K^{\text{Cont}}$ and thus 
    \begin{align*}
         SH^{-,>L,-n+1}_M(K) = H\left(\widehat{\varinjlim_{H \in \mathcal{H}^{\text{Cont}}_K}} CF_w^{S^1,-,>L,-n+1}(H)\right) = 0,
    \end{align*}
    which proves \eqref{sur}.\\
    
    \noindent
    (b) Let $f' : M \to \RR$ be a Morse function on $K$ without critical points of odd Morse index. We may assume that $f'$ is negative because we can translate it when needed. Also, $f'$ can be assumed to be $C^2$-small by multiplying sufficiently small positive number if needed. Consider a small neighborhood $\partial K \times (-\delta,\delta)$ of $\partial K$. There exists a smooth bump function $\psi : K \cup (\partial K \times (-\delta,\delta)) \to [0,1]$ defined by
    \begin{align*}
        \psi = \begin{cases}
            1 &\text{on} \,\,K - (\partial K \times (-\delta,0])\\
            0 &\text{on} \,\, \partial K\times [0,\delta)\\
            \end{cases}
    \end{align*}
    Then $\text{Crit}(\psi f') = \text{Crit}(f)$ and $\partial K$ is a level set of $\psi f'$. Denoting $\psi f'$ by $f$, the function $f : K \to \RR$ satisfies that 
    \begin{itemize}
        \item it is negative and $C^2$-small,\\
        \item it has no critical point of odd Morse index, and\\
        \item $f|_{\partial K} = 0$.\\
    \end{itemize}
    Let $\{H_k\}$ be a cofinal sequence of $\mathcal{H}_K^{\text{Cont}}$. Moreover, we may assume that $$H_k(t,p) = \frac{1}{k} f(p)$$ on $S^1 \times K$. Then every Hamiltonian orbit of $H_k$ on $K$ is a critical point of $f$. Extracting the part that we need in the exact triangle \eqref{gysin}, we have
    \begin{align*}
        \cdots \lr SH^{S^1,>L,-n+1}_M(K;\Lambda) \lr SH^{>L,-n}_M(K;\Lambda) \xlongrightarrow{\xi^{L,-n}} SH^{S^1,>L,-n}_M(K;\Lambda) \lr \cdots.
    \end{align*}
    To prove that $\xi^{L,-n}$ is injective, it suffices to show that 
    \begin{align}\label{mola}
     SH^{S^1,>L,-n+1}_M(K;\Lambda) = 0.   
    \end{align}
    Let 
    \begin{align*}
        u^{\ell} \otimes x \in \Lambda_{\geq 0}[u] \otimes CF_w(H_k)
    \end{align*}
    be any generator of $ CF_w^{S^1,>L,-n+1}(H_k)$. Then     \begin{align*}
        |u^{\ell} \otimes x| = -2\ell + \mu_{\text{CZ}}(x) = -n+1
    \end{align*}
    and hence 
    \begin{align}\label{odd}
        \mu_{\text{CZ}}(x) = -n +2\ell +1.
    \end{align}
 We have two possible cases as follows.
    \begin{enumerate}[label=\arabic*.]
        \item If $x$ is a constant orbit in $K$, then its Conley-Zehnder index is given by
        \begin{align*}
            \mu_{\text{CZ}}(x) = \text{ind}(x;H_k) - n 
        \end{align*}
        as explained in (c) of Remark \ref{morse}. The identity \eqref{odd} becomes 
        \begin{align*}
            \text{ind}(x;H_k) = 2\ell+1.
        \end{align*}
        But this cannot happen due to the choice of $H_k$.\\
        \item If $x$ is a nonconstant orbit outside of $K$, then its Conley-Zehnder index satisfies $\mu_{\text{CZ}}(x) \leq -n$ because $x$ travels in the opposite direction of Reeb orbits on $(\partial K, \alpha)$. From the identity \eqref{odd}, we have the absurd inequality $2\ell+1 \leq 0$. \\
       
    \end{enumerate}
    The discussion of the possible cases 1 and 2 above implies that 
    \begin{align*}
         CF_w^{S^1,>L,-n+1}(H_k) = 0
    \end{align*}
    for each $k=1,2,3,\cdots$ and therefore
    \begin{align*}
        SH^{S^1,>L,-n+1}_M(K;\Lambda) = H\left(\widehat{\varinjlim_{k\to \infty}} CF_w^{S^1,>L,-n+1}(H_k)\right) = 0.
    \end{align*}
    This proves \eqref{mola}.
    \\
\end{proof}
The proof of Lemma \ref{inj} also proves the following statement.
\begin{corollary}
    Let $(M,\omega)$ be a $2n$-dimensional closed symplectically aspherical symplectic manifold and $K\subset M$ be a Liouville domain with index-bounded boundary. Suppose that every contractible Reeb orbit $\gamma$ of $(\partial K, \alpha)$ satisfies $\mu_{\text{CZ}} (\gamma) \geq n$.
    \begin{enumerate}[label=(\alph*)]
        \item $SH^{S^1,-,>L,*}_M(K) = 0$ for $* \geq -n+1$.
        \item If, furthermore, there exists a Morse function on $K$ which does not admit critical points of odd Morse index, then $SH^{S^1,>L,*}_M(K) = 0$ for $* \geq -n+1$.
    \end{enumerate}\qed
\end{corollary}

\begin{theorem}\label{slight}
Let $(M,\omega)$ be a $2n$-dimensional symplectically aspherical closed symplectic manifold and $K\subset M$ be a Liouville domain with index-bounded boundary. If every contractible Reeb orbit $\gamma$ of $(\partial K, \alpha)$ satisfies $\mu_{\text{CZ}} (\gamma) \geq n$, then
     \begin{align}\label{shgh}
         c^{SH}(M,K) \leq c_1^{GH}(M,K)
     \end{align}
     and therefore, together with \eqref{inq},
     \begin{align*}
         c^{SH}(M,K) = c_1^{GH}(M,K).
     \end{align*}
\end{theorem}
\begin{proof}
To prove the inequality \eqref{shgh}, we should show that 
\begin{align}\label{www}
    j^{S^1,L}(1\otimes [K, \partial K])=0 \Longrightarrow j^{L}([K, \partial K])=0.
\end{align}
We can combine two exact triangles in Theorem \ref{ing} to make a bigger commutative diagram as follows.
\begin{align}\label{combine}
    \includegraphics[scale=1.2]{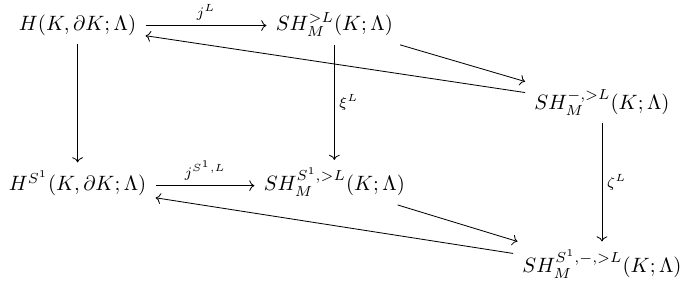}
\end{align}
Every vertical map in \eqref{combine} is induced by $x \mapsto 1\otimes x$ on the chain level. Therefore, the first vertical map $$H(K,\partial K; \Lambda) \lr H^{S^1}(K,\partial K; \Lambda)$$ of \eqref{combine} maps $x$ to $1\otimes x$ and the second vertical map of \eqref{combine} is $$\xi^L : SH^{>L}_M(K;\Lambda) \lr SH^{S^1, >L}_M(K;\Lambda) $$ given in \eqref{gysin}. To see \eqref{www}, we focus on the part of \eqref{combine} to have the following commutative diagram. 
\begin{align}\label{comp}
     \includegraphics[scale=1.2]{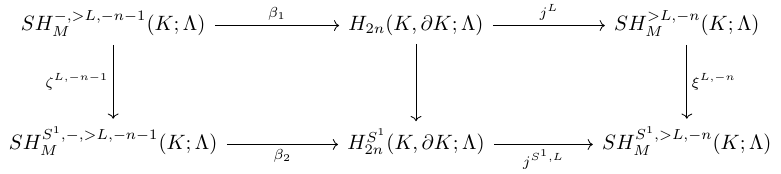}
\end{align}
For convenience, we label the first map on each row of \eqref{comp} by $\beta_1$ and $\beta_2$, respectively.  If $j^{S^1,L}(1 \otimes [K,\partial K]) = 0$, then we can choose $a \in SH^{S^1,-,>L,-n-1}(K;\Lambda)$ such that $$\beta_2 (a) = 1 \otimes [K,\partial K]$$ because of the exactness of the second row. By (a) of Lemma \ref{inj}, the map $\zeta^{L,-n-1}$ is surjective and there exists $a' \in SH^{-,>L,-n-1}_M(K;\Lambda)$ such that $$\zeta^{L,-n-1} (a') = a.$$ The commutativity of the first square of \eqref{comp} implies that $\beta_1(a') = [K,\partial K]$. Therefore, $j^L([K,\partial K]) = 0$ since the first row of \eqref{comp} is also exact.
\\
\end{proof}
\begin{remark}
    In the proof of Theorem \ref{slight}, the surjectivity of $\zeta^{L,-n-1}$ plays an important role. We want to point out that the injectivity of $\xi^{L,-n}$ can be used to prove $c^{SH}(M,K) \leq c^{GH}_1(M,K)$. More precisely, if $j^{S^1,L}(1 \otimes [K,\partial K]) = 0$, then $$\xi^{L,-n}(j^L([K, \partial K])) = j^{S^1,L}(1 \otimes [K,\partial K])=0$$ by the commutativity of the second square of \eqref{comp}. Assuming the injectivity of $\xi^{L,-n}$, we can derive that $j^L([K,\partial K]) = 0$ and hence the desired inequality. In view of Lemma \ref{inj}, we can have the following assertion, which is weaker than Theorem \ref{slight}: Let $(M,\omega)$ be a $2n$-dimensional closed symplectically aspherical symplectic manifold and $K\subset M$ be a Liouville domain with index-bounded boundary. If every contractible Reeb orbit $\gamma$ of $(\partial K, \alpha)$ satisfies $\mu_{\text{CZ}} (\gamma) \geq n$ and there exists a Morse function with critical points of odd Morse index, then $c^{SH}(M,K) \leq c_1^{GH}(M,K)$.
\end{remark}
Theorem \ref{slight} still holds for the non-relative case because the exact triangle \eqref{gysin} in Theorem \ref{rbo} is actually a relative version of exact triangle constructed in \cite{bo} concerning (usual) symplectic cohomology $SH(K;\Lambda)$ and $S^1$-equivariant cohomology $SH^{S^1}(K;\Lambda)$. Also, the proofs of Lemma \ref{inj} and Theorem \ref{slight} still hold if we disregard the subscript $_M$. Note that in the non-relative case, we don't need the index-boundedness assumption on the boundary because admissible Hamiltonian functions to define $SH(K;\Lambda)$ and $SH^{S^1}(K;\Lambda)$ are chosen to be linear at infinity. For a Liouville domain $K$ with $c_1(TK)=0$, the contact structure $\xi$ on $(\partial K, \alpha)$ also satisfies that $c_1(\xi) = 0$ because $c_1(\xi) = i^*c_1(TK)$ where $i: \partial K \hookrightarrow K$ is the inclusion map. So, Reeb orbits on $(\partial K, \alpha)$ have well-defined $\ZZ$-valued Conley-Zehnder indices.

\begin{corollary}\label{nonrela}
    Let $K$ be a $2n$-dimensional Liouville domain with $c_1(TK)=0$. If every contractible Reeb orbit $\gamma$ of $(\partial K, \alpha)$ satisfies $\mu_{\text{CZ}} (\gamma) \geq n$, then
    \begin{align*}
        c_1^{GH}(K) = c^{SH}(K). 
    \end{align*}
    \qed
\end{corollary}

\Addresses
\end{document}